\documentclass[11pt,a4paper]{article}
\usepackage[T5]{fontenc}
\usepackage{fullpage}
\usepackage{amsfonts,amsmath,amssymb,amsthm}
\usepackage{graphicx,xcolor}
\usepackage{mathtools,thmtools}
\usepackage{hyperref,cleveref}

\theoremstyle{plain}
\newtheorem{theorem}{Theorem}
\newtheorem{lemma}{Lemma}
\newtheorem{proposition}{Proposition}
\newtheorem{remark}{Remark}

\theoremstyle{definition}
\newtheorem{algorithm}{Algorithm}

\usepackage{enumerate,enumitem}
\setlist[enumerate]{label=$\left(\roman*\right)$, wide=0pt}

\newcommand{\demi}{\frac{1}{2}}
\newcommand{\optval}{f_{\star}}

\newcommand{\bR}{\mathbb{R}}

\newcommand{\cE}{\mathcal{E}}
\newcommand{\cH}{\mathcal{H}}

\newcommand{\argmin}[2]{\mathrm{argmin}_{#1} #2}
\newcommand{\prox}[2]{\mathsf{prox}_{#1} \left( #2 \right)}

\newcommand{\func}[2]{#1 \left( #2 \right)}
\newcommand{\abs}[1]{\left\lvert #1 \right\rvert}
\newcommand{\innp}[2]{\left\langle #1 , #2 \right\rangle}
\newcommand{\norm}[1]{\left\lVert #1 \right\rVert}

\date{}

\title{The Iterates of Nesterov's Accelerated Algorithm Converge in The Critical Regimes}
\author{Radu Ioan Bo\c{t}\thanks{Faculty of Mathematics, University of Vienna, Oskar-Morgenstern-Platz 1, 1090 Vienna, Austria, e-mail: \url{radu.bot@univie.ac.at}.}  \and Jalal Fadili \thanks{Normandie Universit\'e, ENSICAEN, CNRS, GREYC, France, e-mail:\url{Jalal.Fadili@ensicaen.fr}.} \and Dang-Khoa Nguyen\thanks{Faculty of Mathematics and Computer Science, University of Science, Ho Chi Minh City, Vietnam, e-mail: \url{ndkhoa@hcmus.edu.vn}.} \thanks{Vietnam National University, Ho Chi Minh City, Vietnam.}}
\begin{document}
\maketitle

\begin{abstract}
In this paper, we prove that the iterates of the accelerated Nesterov's algorithm in the critical regime do converge in the weak topology to a global minimizer of an $L$-smooth function in a real Hilbert space, hence answering positively a conjecture posed by H. Attouch and co-authors a decade ago. This result is the algorithmic case of a very recent result on the continuous-time system posted by E. Ryu on X, with assistance from ChatGPT.
\end{abstract}

\section{Introduction}
\subsection{Problem statement}
Throughout the paper, $\cH$ is a real Hilbert space which is endowed with the scalar product $\innp{\cdot}{\cdot}$ and associated norm $\norm{\cdot}$. We consider the optimization problem
\begin{equation}\label{eq:minP}
\min_{x \in \cH} f(x) ,
\end{equation}
where $f : \cH \to \bR$ is a convex and $L_{\nabla f}$-smooth function, i.e. continuously differentiable with $L_{\nabla f}$-Lipschitz gradient. We denote $\argmin{\cH}{f}$ the solution set of \eqref{eq:minP}, which is assumed to be nonempty, and $\optval := \inf_{x \in \cH} f(x) \in \bR$ the optimal value. 

Since the seminal work of Nesterov \cite{Nesterov} (and even before by Gelfand and Tseslin \cite{GT}), accelerated dynamics based on asymptotically vanishing damping, both in continuous-time and their discrete counterparts, have attracted a tremendous interest in the optimization and machine learning literature. For instance, consider the second-order differential equation
\begin{equation}\label{eq:AVD}
\ddot{x}(t) + \frac{\alpha}{t} \dot{x}(t)  +  \nabla f (x(t)) = 0,
\end{equation}
where $\alpha$ is a positive parameter. As a specific feature, the viscous damping coefficient $\frac{\alpha}{t}$ vanishes (tends to zero) as time $t$ goes to infinity, hence the terminology. 

For $\alpha=3$, this system was introduced in \cite{SBC}. There, it was also shown to be a continuous version of the accelerated gradient method of Nesterov \cite{Nesterov}. Its adaptation to the case of structured ``smooth + nonsmooth'' convex optimization gives the Fast Iterative Shrinkage-Thresholding Algorithm (FISTA) of \cite{BT:09}. The convergence properties of the dynamic \eqref{eq:AVD} have been the subject of a large body of work; see e.g., \cite{AAD,AC1,AC:18,AC2R-EECT,ACPR,ACR-subcrit,AP,AD,AD17,May,SBC}. In particular, for $\alpha \geq 3$, the generated trajectory $x(t)$ of \eqref{eq:AVD} satisfies the asymptotic rate of convergence of the values $f(x(t)) - f_* =O\left(1 /t^2\right)$ as $t \to +\infty$. For $\alpha > 3$, it was shown in \cite{AP} and \cite{May}, that the trajectory converges weakly to a minimizer of $f$, with the improved rate of convergence for the function values along the trajectory convergence of $o(1/t^2)$ as $t \to +\infty$. 

Corresponding results for the algorithmic case have been obtained by \cite{CD:15} and \cite{AP}. 
Recently, it was shown in \cite{ABHN} that there is some connection between \eqref{eq:AVD} and the Heavy-ball method of Polyak \cite{Polyak1964}.

For the critical case $\alpha=3$, the (weak) convergence of the trajectory for the continuous system \eqref{eq:AVD}, and the iterates for the discrete Nesterov algorithm has remained elusive and a widely open problem. It was only proved, both for the continuous dynamics and Nesterov algorithm, in \cite{APR} and \cite{ACR-subcrit} either under restrictive assumptions on the function $f$, or in the one-dimensional case without any restriction on the parameters.

In a post on X by E. Ryu \cite[Part~I]{Ryu} on October 21, the author intensively relied on interaction with ChatGPT to provide a surprisingly simple proof of convergence of the trajectories of \eqref{eq:AVD} in the critical case and finite dimensional spaces, hence settling the convergence of the trajectory of \eqref{eq:AVD}. The discrete algorithmic case was mentioned in a subsequent post, but had not been solved by the time the first version of this note was submitted. The discrete case was subsequently and independently handled in \cite[Part~III]{Ryu} and \cite{JR:25} in finite dimension.  The convergence analysis in \cite[Part~III]{Ryu} is unnecessarily convoluted, and was later substantially simplified in \cite{JR:25}. This simplification is due to the use of Lemma~A.4 from \cite{BCCH}\footnote{We had brought this lemma to the attention of Jiang and Ryu in a private communication.
}

\subsection{Contributions}

In this paper, we consider the following general accelerated algorithm for solving \eqref{eq:minP}:
\begin{algorithm}
\label{algo:FISTA}
Let $\left( \alpha_{k} \right) _{k \geq 1} \subseteq [0, +\infty)$ and $0 < \lambda \leq \frac{1}{L_{\nabla f}}$, and $x_{0}, x_{1}$ be given in $\cH$.
For every $k \geq 1$, set
\begin{equation*}
\begin{dcases}
y_{k} 	& \coloneq x_{k} + \alpha_{k} \left( x_{k} - x_{k-1} \right) , \\
x_{k+1}	& \coloneq y_{k} - \lambda \nabla \func{f}{y_{k}}.
\end{dcases}
\end{equation*}
\end{algorithm}
The sequence $\left( \alpha_{k} \right) _{k \geq 1}$ is constructed, for every $k \geq 1$, as follows
\begin{equation*}
\alpha_{k} \coloneq \dfrac{t_{k}-1}{t_{k+1}} ,
\end{equation*}
with the sequence $\left( t_{k} \right) _{k \geq 1} \subseteq (0, +\infty)$ chosen appropriately.

Our setting is general enough so that this algorithm covers the critical cases corresponding to both the Nesterov algorithm \cite{Nesterov} and that of Chambolle-Dossal in \cite{CD:15}. We show that the iterates of the accelerated Nesterov's algorithm in the critical regime do converge in the weak topology to a minimizer of $f$, hence solving the conjecture posed by Attouch and co-authors a decade ago. This complements the result of \cite{Ryu} in the continuous-time case. This result also holds for the smooth+nonsmooth case when using FISTA. The crux of the proof relies on three tenets: (i) properties of $(t_k)_{k \geq 1}$; (ii) 
an elegant representation of our sequence $(x_k)_{k \geq 0}$ as an iterative average of a bounded sequence, allowing us to transfer this boundedness property to $(x_k)_{k \geq 0}$ via Jensen’s inequality; and (iii) a proof technique reminiscent of that of Opial's lemma without directly invoking the latter as usually done.

\section{Convergence proof}
\subsection{Properties of the sequence $\left( t_{k} \right)_{k \geq 1}$}

Let $0 \leq \theta < 1$. We will assume that the sequence $\left( t_{k} \right) _{k \geq 0} \subseteq (0, +\infty)$ is chosen such that
\begin{equation}
\label{defi:t}
t_{1} =1 \quad \mbox{and} \quad t_{k+1}^{2} - t_{k}^{2} = \left( 1 - \theta \right) t_{k+1} + \theta t_{k} \quad \forall k \geq 1.
\end{equation}

\begin{remark}\label{rem:tk}
\begin{enumerate}
\item 
The classical choice of Nesterov in \cite{Nesterov},
\begin{equation}\label{nes}
t_1 = 1 \quad \mbox{and} \quad t_{k+1} := \frac{1 + \sqrt{1+4t_k^2}}{2} \quad \forall k \geq 1,
\end{equation}
thus satisfying 
$$t_{k+1}^{2} - t_{k}^{2} = t_{k+1} \quad \forall k \geq 1,$$
corresponds to $\theta = 0$.

\item 
The choice of Chambolle-Dossal in \cite{CD:15}, 
\begin{equation}\label{cd}
t_{k} \coloneq 1 + \dfrac{k-1}{\alpha-1} \quad \forall k \geq 1,
\end{equation}
satisfies \eqref{defi:t} in the case $\alpha = 3$ for $\theta = \demi$, namely,
\begin{equation*}
t_{k+1}^{2} - t_{k}^{2} = \left( t_{k+1} - t_{k} \right) \left( t_{k+1} + t_{k} \right) = \dfrac{1}{\alpha-1} \left( t_{k+1} + t_{k} \right) = \demi \left( t_{k+1} + t_{k} \right) \quad \forall k \geq 1 .
\end{equation*}
\end{enumerate}
These correspond to the two critical regimes of the accelerated gradient method. While it has long been known that $f(x_k) - \optval = O(1/k^2)$ as $k \to +\infty$, in both cases, the convergence of the sequence of iterates $(x_k)_{k \geq 0}$ has remained an open question -- the one that we answer affirmatively in this note. 
\end{remark}

\begin{lemma}
\label{lem:t}
Let $\left( t_{k} \right)_{k \geq 1}$ be a sequence defined according to \eqref{defi:t}. For every $k \geq 1$, it holds:
\begin{enumerate}
\item
\label{lem:t_claim0}	
$t_{k+1} = \frac{ 1-\theta + \sqrt{\left( 1 - \theta \right) ^{2} + 4 \left( t_{k}^{2} + \theta t_{k} \right)} }{2}$.
\item
\label{lem:t_claim1}
$t_{k} \geq \frac{\left( 1 - \theta \right)}{2} \left( k + 1 \right)$. 
\end{enumerate}
\end{lemma}
\begin{proof}
The explicit formula of $t_{k+1}$ in \ref{lem:t_claim0} comes from the fact that it is the positive solution of the quadratic equation
\begin{equation*}
t^{2} - \left( 1 - \theta \right) t - \left( t_{k}^{2} + \theta t_{k} \right) = 0 .
\end{equation*}
We show the inequality in \ref{lem:t_claim1} by induction. The inequality is satisfied in the case $k=1$. Suppose that it is also satisfied for some $k \geq 1$, namely, $2t_{k} \geq \left( 1 - \theta \right) \left( k + 1 \right)$. 
Then, using the explicit formula of $t_{k+1}$ in \ref{lem:t_claim0}, we see that
\begin{align*}
t_{k+1} = \frac{ 1-\theta + \sqrt{\left( 1 - \theta \right) ^{2} + 4 \left( t_{k}^{2} + \theta t_{k} \right)} }{2}
& \geq \frac{ 1-\theta + \sqrt{4 t_{k}^{2}} }{2} \nonumber \\
& \geq \dfrac{1-\theta + \left( 1 - \theta \right) \left( k + 1 \right)}{2}
= \frac{\left( 1 - \theta \right) \left( k + 2 \right)}{2}.
\end{align*}
\end{proof}

\subsection{Lyapunov analysis and boundedness of the iterates}

We introduce the auxiliary sequence $\left( z_{k} \right) _{k \geq 1}$, defined, for every $k \geq 1$, as
\begin{equation}
\label{defi:z}
z_{k} \coloneq t_{k} \left( x_{k} - x_{k-1} \right) + x_{k-1} . 
\end{equation}
For every $k \geq 1$, it holds
\begin{align}\label{defi:z:dt}
z_{k} = \left( t_{k} - 1 \right) \left( x_{k} - x_{k-1} \right) + x_{k} = t_{k} x_{k} - \left( t_{k} - 1 \right) x_{k-1}.
\end{align}
This representation will lead to the interpretation of $\left( x_{k} \right) _{k \geq 1}$ as an ergodic sequence, as in \cite{BCN:23}. To this end, we define the sequences $\left( s_{k} \right) _{k \geq 1}$ and $\left( u_{k} \right) _{k \geq 1}$, for every $k \geq 1$ as follows
\begin{align}
s_{k}	& \coloneq t_{k}^{2} + \theta t_{k} , \label{defi:s} \\
u_{k}	& \coloneq z_{k} + \theta \left( x_{k} - x_{k-1} \right) . \label{defi:u}
\end{align}

By construction, for every $k \geq 1$, we have
\begin{equation}
\label{lem:t_claim2} 
s_{k} =  \theta + \sum_{i=1}^{k} t_{i} .
\end{equation}
Indeed, for every $k \geq 1$, according to \eqref{defi:t} it holds
\begin{equation*}
t_{k+1}^{2} + \theta t_{k+1} = t_{k}^{2} + \theta t_{k} + t_{k+1}.
\end{equation*}
In the view of \eqref{defi:s}, is nothing else than
\begin{equation*}
s_{k+1} = s_{k} + t_{k+1} .
\end{equation*}
By using the telescoping sum argument, we obtain, for every $k \geq 1$,
\begin{equation*}
s_{k+1} = s_{1} + \sum_{i=1}^{k} t_{i+1} = 1 + \theta + \sum_{i=2}^{k+1} t_{i} = \theta + \sum_{i=1}^{k+1} t_{i} .
\end{equation*}

\begin{lemma}
Let $\left( t_{k} \right)_{k \geq 1}$ be a sequence defined according to \eqref{defi:t}, and $\left( x_{k} \right) _{k \geq 0}$ be the sequence generated by \Cref{algo:FISTA}.
For every $k \geq 1$, it holds
\begin{equation}
\label{defi:x}
x_{k} = \dfrac{1}{s_{k}} \left(\theta x_{0} + \sum_{i=1}^{k} t_{i} u_{i} \right)
= \dfrac{1}{\sum_{i=0}^{k} t_{i}} \sum_{i=0}^{k} t_{i} u_{i} ,
\end{equation}
with the convention
\begin{equation*}
t_{0} \coloneq \theta
\text{ and }
u_{0} \coloneq x_{0} .
\end{equation*}
\end{lemma}
\begin{proof}
Firstly, note that
\begin{align}
t_{1} u_{1}
& = x_{1}  + \theta \left( x_{1} - x_{0} \right) = (1+\theta)x_1 - \theta x_0 = s_{1} x_{1} - \theta x_{0} , \label{eq:init}
\end{align}    
which means \eqref{defi:x} holds for $k = 1.$

Now, consider \eqref{defi:z:dt} at the index $k+1$, for $k \geq 1$. By multiplying both sides by $t_{k+1} > 0$, we obtain
\begin{align*}
t_{k+1} z_{k+1}  
& = t_{k+1}^{2} x_{k+1} - \left( t_{k+1}^{2} - t_{k+1} \right) x_{k} \nonumber \\
& = t_{k+1}^{2} x_{k+1} - \left( t_{k}^{2} + \theta \left( t_{k} - t_{k+1} \right) \right) x_{k} \nonumber \\
& = \left( t_{k+1}^{2} + \theta t_{k+1} \right) x_{k+1} - \left( t_{k}^{2} + \theta t_{k} \right) x_{k} - \theta t_{k+1} \left( x_{k+1} - x_{k} \right) .
\end{align*}
In the view of $s_{k}$, we can rewrite the relation above as
\begin{equation*}
t_{k+1} z_{k+1} + \theta t_{k+1} \left( x_{k+1} - x_{k} \right) = \left( t_{k+1}^{2} + \theta t_{k+1} \right) x_{k+1} - \left( t_{k}^{2} + \theta t_{k} \right) x_{k} = s_{k+1} x_{k+1} - s_x x_k.
\end{equation*}
By using the telescoping sum argument, we obtain
\begin{equation*}
\sum_{i=1}^{k} t_{i+1} u_{i+1} = s_{k+1} x_{k+1} - s_{1} x_{1},
\end{equation*}
which gives
\begin{align*}
s_{k+1} x_{k+1} & = s_{1} x_{1} + \sum_{i=1}^{k} t_{i+1} u_{i+1} 
= s_{1} x_{1} - t_{1} u_{1} + \sum_{i=1}^{k+1} t_{i} u_{i} =  x_{0} + \sum_{i=1}^{k+1} t_{i} u_{i},
\end{align*}
where the last equation comes from \eqref{eq:init}.
\end{proof}

\begin{remark}
\begin{enumerate}
\item 
In the original setting of Nesterov, which corresponds to the case when $\theta = 0$, the sequences $\left( z_{k} \right) _{k \geq 1}$ and $\left( u_{k} \right) _{k \geq 1}$ coincide.

\item 
To emphasize the averaging property and to simplify the latter proof, we see that, for every $k \geq 1$, we can write
\begin{equation*}
x_{k} = \sum_{i=0}^{k} \theta_{k,i} u_{i} ,
\end{equation*}
where, for every $i = 0 , 1 , 2 , \dotsc , k$,
\begin{equation*}
\theta_{k,i} \coloneq \dfrac{t_{i}}{s_{k}}
\end{equation*}
is such that $\theta_{k,i} \geq 0$ and $\sum_{i=0}^{k} \theta_{k,i} = 1$.
\end{enumerate}
\end{remark}

Let $z \in \argmin{\cH}{f}$. For every $k \geq 1$, we denote
\begin{align}
W_{k} 	& \coloneq \func{f}{x_{k}} - \optval + \demi \norm{x_{k} - x_{k-1}}^{2} , \label{defi:W} \\
\cE_{z,k} 	& \coloneq t_{k}^{2} \left( \func{f}{x_{k}} - \optval \right) + \demi \norm{z_{k} - z}^{2} . \label{defi:E}
\end{align}

Using that the sequence $\left( t_{k} \right) _{k \geq 1}$ is nondecreasing, we have, for every $k \geq 1$, 
$$t_{k+1} + \theta \left( t_{k} - t_{k+1} \right) \leq t_{k+1} \quad \forall k \geq 1,$$
therefore,
\begin{equation}\label{cond:t}
t_1 = 1 \quad \mbox{and} \quad t_{k+1}^{2} - t_{k}^{2} \leq t_{k+1} \quad \forall k \geq 1.
\end{equation}

Since condition \eqref{cond:t} is fulfilled, we can use existing results from the literature to state the following proposition, see for instance \cite{BT:09}, \cite[Theorem 2]{CD:15} or \cite[Propositions 3 and 6]{AC:18}.

This means the condition \eqref{cond:t} is verified allows us to exploit these results.
\begin{proposition}
\label{prop:E}
Let $z \in \argmin{\cH}{f}$, $\left( t_{k} \right)_{k \geq 1}$ be a sequence defined according to \eqref{defi:t}, and $\left( x_{k} \right) _{k \geq 0}$ be the sequence generated by \Cref{algo:FISTA}. Then, the following statements are true:
\begin{enumerate}
\item 
The sequences $\left( W_{k} \right) _{k \geq 1}$ and $\left( \cE_{z,k} \right) _{k \geq 1}$ are nonincreasing and thus converge.

\item 
For every $k \geq 1$, it holds
\begin{equation*}
0 \leq \func{f}{x_{k}} - \optval \leq \dfrac{\cE_{z,1}}{t_{k}^{2}} \leq \dfrac{4 \cE_{z,1}}{\left(1-\theta\right)^{2} \left(k+1\right)^{2}} .
\end{equation*}
\end{enumerate}
\end{proposition}

The boundedness of the sequence $\left( x_{k} \right) _{k \geq 0}$ can easily be deduced from the above considerations. Note that the boundedness claim for the Chambolle-Dossal choice, see  \eqref{cd} with $\alpha=3$, was shown in \cite[Proposition~4.3]{ACR-subcrit}, although with a completely different proof.
\begin{proposition}
\label{prop:bound}
Let $\left( t_{k} \right)_{k \geq 1}$ be a sequence defined according to \eqref{defi:t}, and $\left( x_{k} \right) _{k \geq 0}$ be the sequence generated by \Cref{algo:FISTA}.
Then, the following statements are true:
\begin{enumerate}
\item 
\label{prop:bound:x}
The sequence $\left( x_{k} \right) _{k \geq 0}$ is bounded.
\item 
\label{prop:bound:vel}
For every $k \geq 1$, it holds
\begin{equation}
\label{bound:vel}
\norm{x_{k}-x_{k-1}} \leq \dfrac{\sqrt{2} \left( 2 \sqrt{\cE_{z,1}} + \theta \sqrt{W_{1}} \right)}{t_{k}} .
\end{equation}
\item 
\label{prop:bound:clus}
Every weak sequential cluster point of the sequence $\left( x_{k} \right) _{k \geq 0}$ belongs to $\argmin{\cH}{f}$.
\end{enumerate}
\end{proposition}
\begin{proof}
Let $z \in \argmin{\cH}{f}$. It follows immediately from \Cref{prop:E} that the sequences $\left( z_{k} \right) _{k \geq 1}$ and $\left( x_{k} - x_{k-1} \right) _{k \geq 1}$ are bounded. Therefore, the sequence $\left( u_{k} \right) _{k \geq 1}$, defined in \eqref{defi:u}, is also bounded. Precisely, for every $k \geq 1$, we have
\begin{align*}
\norm{z_{k}-z} \leq \sqrt{2 \cE_{z,1}} < + \infty
\textrm{ and }
\norm{x_{k}-x_{k-1}} \leq \sqrt{2 W_{1}} < + \infty ,
\end{align*}
which implies
\begin{equation*}
\norm{u_{k}-z} \leq \norm{z_{k}-z} + \theta \norm{x_{k}-x_{k-1}} \leq C_{z} \coloneq \sqrt{2 \cE_{z,1}} + \theta \sqrt{2 W_{1}} < + \infty .
\end{equation*}
Using the convexity of the norm and applying Jensen's inequality, we will show that the sequence $\left( x_{k} \right)_{k \geq 0}$ is also bounded. Indeed, for every $k \geq 1$, we have
\begin{align*}
\norm{x_{k}-z} & = \norm{ \sum_{i=0}^{k} \theta_{k,i} \left( u_{i} - z \right)}
\leq \sum_{i=1}^{k+1} \theta_{k,i} \norm{u_{i}- z}
\leq \sum_{i=1}^{k+1} \theta_{k,i} C_{z} = C_{z} < + \infty .
\end{align*}
Combining the boundedness of $\left( z_{k} \right) _{k \geq 1}$ and $\left( x_{k} \right) _{k \geq 0}$, we can deduce the convergence rate for discrete velocity.
According to the equation \eqref{defi:z}. For every $k \geq 1$, we have
\begin{equation*}
t_{k} \norm{x_{k}-x_{k-1}} \leq \norm{z_{k}-z} + \norm{x_{k-1}-z} \leq \sqrt{2 \cE_{z,1}} + C_{z} ,
\end{equation*}
which is nothing else than \eqref{bound:vel}.

The existence of weak cluster point of sequence $\left( x_{k} \right) _{k \geq 0}$ is guaranteed by its boundedness.
The fact that $\lim_{k \to + \infty} \func{f}{x_{k}} = \optval$ together with the continuity of $f$ ensures that every weak sequential cluster point of $\left( x_{k} \right) _{k \geq 0}$ belongs to $\argmin{\cH}{f}$.
\end{proof}

\subsection{Weak convergence of the iterates}
For every $z \in \argmin{\cH}{f}$ and $k \geq 1$, we denote
\begin{equation}
\label{defi:h}
h_{z,k} \coloneq \demi \norm{x_{k} - z}^{2} .
\end{equation}
In the view of this notation, we have, for every $k \geq 1$,
\begin{align}
\cE_{z,k} 
& = t_{k}^{2} \left( \func{f}{x_{k}} - \optval \right) + \demi \norm{\left( t_{k} - 1 \right) \left( x_{k} - x_{k-1} \right) + x_{k} - z}^{2} \nonumber \\
& = t_{k}^{2} \left( \func{f}{x_{k}} - \optval \right) + \demi \left( t_{k} - 1 \right)^{2} \norm{x_{k} - x_{k-1}}^{2} + \left( t_{k} - 1 \right) \innp{x_{k} - x_{k-1}}{x_{k} - z} + h_{z,k} . \label{eq:Eh}
\end{align}

\begin{theorem}\label{maintheorem}
Let $\left( t_{k} \right)_{k \geq 1}$ be a sequence defined according to \eqref{defi:t}, and $\left( x_{k} \right) _{k \geq 0}$ be the sequence generated by \Cref{algo:FISTA}. Then, the sequence $\left( x_{k} \right) _{k \geq 0}$ converges weakly to a minimizer of $f$.
\end{theorem}
\begin{proof}
Let $z$ and $v$ be two arbitrary weak sequential cluster point of $\left( x_{k} \right) _{k \geq 0}$. This means that there exists two subsequences $\left( x_{m_{k}} \right) _{k \geq 0}$ and $\left( x_{n_{k}} \right) _{k \geq 0}$ of $\left( x_{k} \right) _{k \geq 0}$, which converge weakly to $z$ and $v$, respectively, as $k \to +\infty$.

For every $k \geq 1$, we set
\begin{align}
D_{k} \coloneq \cE_{z,k} - \cE_{v,k} . \label{defi:D}
\end{align}
\Cref{prop:bound} tells us that $z$ and $v$ belong to $\argmin{\cH}{f}$. From \Cref{prop:E}, we conclude that $D_{\star}:=\lim_{k \to + \infty} D_{k}$ exists and is finite.

For every $k \geq 1$, we set
\begin{align}
R_{k} & \coloneq  h_{z,k} - h_{v,k} . \label{defi:R} 
\end{align}
According to \eqref{defi:h}, we have, for every $k \geq 1$,
\begin{align}
R_{k} = h_{z,k} - h_{v,k}
= \demi \left( \norm{x_{k} - z}^{2} - \norm{x_{k} - v}^{2} \right) = \innp{x_{k} - v}{v - z} + \demi \norm{v - z}^{2} . \label{eq:R:x}
\end{align}
From here, we deduce that, for every $k \geq 1$,
\begin{align}
R_{k+1} - R_{k}
& = \left( h_{z,k+1} - h_{v,k+1} \right) - \left( h_{z,k} - h_{v,k} \right) \nonumber \\
& = \left( \innp{x_{k+1} - v}{v - z} + \demi \norm{v - z}^{2} \right) - \left( \innp{x_{k} - v}{v - z} + \demi \norm{v - z}^{2} \right) \nonumber \\
& = \innp{x_{k+1} - x_{k}}{v - z} . \label{eq:dR}
\end{align}
By the Cauchy-Schwarz inequality, we infer
\begin{equation*}
0 \leq \abs{R_{k+1} - R_{k}} = \abs{\innp{x_{k+1} - x_{k}}{v - z}} \leq \norm{x_{k+1} - x_{k}} \norm{v - z} ,
\end{equation*}
hence, \eqref{bound:vel} yields
\begin{equation*}
\lim_{k \to + \infty} \left( R_{k+1} - R_{k} \right) = 0 .
\end{equation*}

Next, by using \eqref{eq:Eh}, we have, for every $k \geq 1$, that
\begin{align}
D_{k} = \cE_{z,k} - \cE_{v,k}
& = \left( t_{k} - 1 \right) \innp{x_{k} - x_{k-1}}{v - z} + \left( h_{z,k} - h_{v,k} \right) \nonumber \\
& = \left( t_{k} - 1 \right) \innp{x_{k} - x_{k-1}}{v - z} + R_{k}, \label{eq:D:E:R}
\end{align}
which, in the light of \eqref{eq:dR}, yields
\begin{align*}
t_{k+1} D_{k+1} 
& = \left( t_{k+1}^{2} - t_{k+1} \right) \innp{x_{k+1} - x_{k}}{v - z} + t_{k+1} R_{k+1} \nonumber \\
& = \left( t_{k+1}^{2} - t_{k+1} \right) \left( R_{k+1} - R_{k} \right) + t_{k+1} R_{k+1} \nonumber \\
& = t_{k+1}^{2} R_{k+1} - \left( t_{k+1}^{2} - t_{k+1} \right) R_{k} \nonumber \\
& = t_{k+1}^{2} R_{k+1} - \left( t_{k}^{2} + \theta \left( t_{k} - t_{k+1} \right) \right) R_{k} .
\end{align*}
Adding $\theta t_{k+1} \left( R_{k+1} - R_{k} \right)$ on both sides, leads, for every $k \geq 1$, to
\begin{equation*}
t_{k+1} \left[ D_{k+1} + \theta \left( R_{k+1} - R_{k} \right) \right] = \left( t_{k+1}^{2} + \theta t_{k+1} \right) R_{k+1} - \left( t_{k}^{2} + \theta t_{k} \right) R_{k}
\end{equation*}
or, equivalently,
\begin{equation*}
s_{k+1} R_{k+1} = s_{k} R_{k} + t_{k+1} \left[ D_{k+1} + \theta \left( R_{k+1} - R_{k} \right) \right] .
\end{equation*}
By the telescoping sum argument, we obtain, for every $k \geq 1$, that
\begin{align*}
s_{k+1} R_{k+1} 
& = s_{1} R_{1} + \sum_{i=1}^{k} t_{i+1} \left[ D_{i+1} + \theta \left( R_{i+1} - R_{i} \right) \right] \nonumber \\
& = s_{1} R_{1} + \sum_{i=2}^{k+1} t_{i} \left[ D_{i} + \theta \left( R_{i} - R_{i-1} \right) \right] \nonumber \\
& = - t_{1} \left[ D_{1} + \theta \left( R_{1} - R_{0} \right) \right] + s_{1} R_{1} + \sum_{i=1}^{k+1} t_{i} \left[ D_{i} + \theta \left( R_{i} - R_{i-1} \right) \right] .
\end{align*}
Since $s_{k+1} > 0$, the identity above can be equivalently written, for every $k \geq 1$, as
\begin{align*}
R_{k+1} 
& = \dfrac{1}{s_{k+1}} \left(\left[ \theta \left( R_{0} - R_{1} \right) - D_{1} \right] + s_{1} R_{1} + \sum_{i=1}^{k+1} t_{i} \left[ D_{i} + \theta \left( R_{i} - R_{i-1} \right) \right] \right) \nonumber \\
& = \dfrac{\theta \left( R_{0} - R_{1} \right) - D_{1}}{s_{k+1}} + \dfrac{s_{1} R_{1} + \sum_{i=1}^{k+1} t_{i} \left[ D_{i} + \theta \left( R_{i} - R_{i-1} \right) \right]}{s_{1} + \sum_{i=1}^{k+1} t_{i}} ,
\end{align*}
where, in the last term of the second equation, we used \eqref{lem:t_claim2}.
From \Cref{lem:t} \ref{lem:t_claim1}, we also know that $\lim_{k \to + \infty} t_{k} = + \infty$, so that 
\begin{equation*}
\lim_{k \to + \infty} \dfrac{\theta \left( R_{0} - R_{1} \right) - D_{1}}{s_{k+1}} = 0 .
\end{equation*}
Finally,  applying Stolz-Ces\`{a}ro Theorem, we obtain
\begin{align*}
\lim_{k \to + \infty} \dfrac{s_{1} R_{1} + \sum_{i=1}^{k+1} t_{i} \left[ D_{i} + \theta \left( R_{i} - R_{i-1} \right) \right]}{s_{1} + \sum_{i=1}^{k+1} t_{i}} 
& = \lim_{k \to + \infty} \dfrac{t_{k+2} \left[ D_{k+2} + \theta \left( R_{k+2} - R_{k+1} \right) \right]}{t_{k+2}} \nonumber \\
& = \lim_{k \to + \infty} \left[ D_{k+2} + \theta \left( R_{k+2} - R_{k+1} \right) \right] = D_{\star} \in \bR .
\end{align*}
From here, we conclude that
\begin{equation*}
\lim_{k \to + \infty} R_{k} = \lim_{k \to + \infty} D_{k} = D_{\star}.
\end{equation*}
Passing to the limit \eqref{eq:R:x} along the two subsequences $\left( x_{m_{k}} \right) _{k \geq 0}$ and $\left( x_{n_{k}} \right)_{k \geq 0}$, we obtain
\begin{equation*}
D_{\star} = - \demi \norm{v-z}^{2}
\text{ and }
D_{\star} = \demi \norm{v-z}^{2} ,
\end{equation*}
respectively. This implies $v = z$, and the proof is completed.
\end{proof}

\begin{remark}
\begin{enumerate}
\item Using that $\lim_{k \to + \infty} \alpha_{k} = 1$ and $\lim_{k \to + \infty} (x_{k} - x_{k-1}) = 0$, from Theorem \ref{maintheorem} we deduce that the sequence $\left( y_{k} \right) _{k \geq 1}$ also converges weakly to a minimizer of f. This means that the sequence generated by the Ravine method  \cite{GT,AF} also converges.

\item For the choice of $(t_k)_{k \geq 1}$ following the Chambolle–Dossal rule \eqref{cd}, it is known (see, for instance, \cite{CD:15, AP}) that when $\alpha > 3$ a \emph{little-o} rate of convergence can be achieved; that is,
\begin{equation*}
\func{f}{x_{k}} - \optval = \func{o}{\frac{1}{t_{k}^{2}}} = \func{o}{\frac{1}{k^{2}}} \text{ as } k \to + \infty .
\end{equation*}
The key ingredients for this improved convergence rate are the following summability properties
\begin{equation}
\label{little-o}
\sum_{k \geq 1} t_{k+1} \left( \func{f}{x_{k}} - \optval \right) < + \infty 
\quad \text{ and } \quad
\sum_{k \geq 1} t_{k} \norm{x_{k}-x_{k-1}}^{2} < + \infty .
\end{equation}
These conditions are also fundamental in establishing the weak convergence of the iterates. However, it remains doubtful whether such summability results continue to hold in the critical case $\alpha = 3$ or under the Nesterov rule \eqref{nes}.

Even in the simpler continuous-time setting considered by Ryu \cite{Ryu}, it is unclear whether analogous integrability results can be derived for \eqref{eq:AVD} when $\alpha =3$.

\item Shortly after we submitted a first version of this note, we realized that Theorem~\ref{maintheorem} remains valid if the sequence $(t_k)_{k \geq 1}$ satisfies the weaker condition \eqref{cond:t}, that is,
\[
t_1=1 \quad \mbox{and} \quad t_{k+1}^2 - t_k^2 \leq t_{k+1} \quad \forall k \geq 1 .
\]
This turns out to hold also for both the Nesterov rule \eqref{nes} and the Chambolle-Dossal rule \eqref{cd} with $\alpha=3$. Taking into account, for instance, \cite[Theorem~2]{CD:15}, we obtain that all statements of Proposition~\ref{prop:E} remain true. Moreover, the sequence $(z_k)_{k \geq 1}$ is bounded. Since, for every $k \geq 1$,
\begin{equation*}
x_{k} = \left( 1 - \dfrac{1}{t_{k}} \right) x_{k-1} + \dfrac{1}{t_{k}} z_{k},
\end{equation*}
it follows by induction (see the independent work of \cite[Part~III]{Ryu}, \cite{JR:25}, assisted by ChatGPT, posted after our initial submission) that
\begin{equation*}
\norm{x_{k}} \leq  \max \left\lbrace \norm{x_{0}} , \sup_{k \geq 1} \norm{z_{k}} \right\rbrace.
\end{equation*}
In other words, the sequence $(x_k)_{k \geq 1}$ is also bounded. From this point, one can proceed with the same arguments as in the proof of Theorem~\ref{maintheorem}. Note that boundedness has been already established by \cite[Proposition~4.3]{ACPR} for $\alpha=3$ in the Chambolle-Dossal rule \eqref{cd}. 

\item In the proof of Theorem~\ref{maintheorem}, from \eqref{eq:D:E:R} we have
\[
D_{k} = R_{k} + (t_k - 1)(R_k - R_{k-1}) \quad \forall k \geq 1.
\]
Since $\sum_{k \geq 1} \frac{1}{t_{k}} = + \infty$ and $\lim_{k \to +\infty} D_k = D_\star$, one could have alternatively invoked \cite[Lemma A.4]{BCCH} to conclude directly that $\lim_{k \to +\infty} R_k = D_\star$.

\end{enumerate}
\end{remark}

\subsection{Accelerated proximal-gradient algorithm}

The extension of Algorithm~\ref{algo:FISTA} to the non-smooth composite optimization problem
\begin{equation*}
\min\limits_{x \in \cH} \left\lbrace \func{f}{x} + \func{g}{x} \right\rbrace ,
\end{equation*}
where $f: \cH \to \bR$ is convex and $L_{\nabla f}$-smooth, and $g \colon \cH \to \bR \cup \left\lbrace + \infty \right\rbrace$ is proper, convex and lower semicontinuous, with nonempty set of minimizers, is the fast proximal-gradient algorithm (see, e.g., \cite{BT:09}). For $\left( \alpha_{k} \right) _{k \geq 1} \subseteq \bR$, $0 < \lambda \leq \frac{1}{L_{\nabla f}}$, and $x_{0}, x_{1} \in \cH$, the iterates are defined, for every $k \geq 1$, by
\begin{equation*}
\begin{dcases}
y_{k} 	& \coloneq x_{k} + \alpha_{k} \left( x_{k} - x_{k-1} \right) , \\
x_{k+1}	& \coloneq \prox{\lambda g}{y_{k} - \lambda \nabla \func{f}{y_{k}}},
\end{dcases}
\end{equation*}
where the proximal point operator associated to $g$ with parameter $\lambda > 0$ is given by
\begin{equation*}
\mathsf{prox}_{\lambda g} \colon \cH \to \cH, \quad \prox{\lambda g}{y} \coloneq \argmin{x \in \cH}{\left\lbrace \func{g}{x} + \dfrac{1}{2 \lambda} \norm{x-y}^{2} \right\rbrace}.
\end{equation*}
If $\left( t_{k} \right)_{k \geq 1}$ is a sequence defined according to \eqref{cond:t}, and therefore also according to \eqref{defi:t}, and
\begin{equation*}
\alpha_{k} \coloneq \dfrac{t_{k}-1}{t_{k+1}} \quad \forall k \geq 1,
\end{equation*}
then, by defining the energy functions in \eqref{defi:W}–\eqref{defi:E} accordingly and by using, for instance, \cite[Theorem 2]{CD:15}, one can analogously show that the generated sequence $(x_k)_{k \geq 0}$ converges weakly to a minimizer of $f+g$.

This result provides a positive answer to the long-standing open question regarding the convergence of the iterates generated by FISTA, namely, for sequences $\left( t_{k} \right)_{k \geq 1}$ following either the Nesterov rule \eqref{nes} as in \cite{BT:09} as well as  the Chambolle-Dossal rule as in \cite{CD:15} with $\alpha=3$.

\section*{Acknowledgement}
We would like to thank Tony Silveti-Falls who draw our attention to the post \cite{Ryu}.



\begin{thebibliography}{99}
\bibitem{AAD} 
\textbf{V. Apidopoulos, J.-F. Aujol, and  Ch. Dossal}, 
{\it The differential inclusion modeling FISTA algorithm and optimality of convergence rate in the case $b\leq 3$}, 
SIAM Journal on Optimization \textbf{28}(1),  551--574 (2018)

\bibitem{ABHN}
\textbf{H. Attouch, R. I. Bo\c{t}, D. A. Hulett, and D.-K. Nguyen},
\textit{Recovering Nesterov accelerated dynamics from Heavy Ball dynamics via time rescaling}. \url{arXiv:2504.15852} (2025) 


\bibitem{AC1}  
\textbf{H. Attouch and A. Cabot}, 
{\it Asymptotic stabilization of inertial gradient dynamics with time-dependent viscosity},  
Journal of Differential Equations 
\textbf{263},  5412--5458  (2017)

\bibitem{AC:18}
\textbf{H. Attouch and A. Cabot},
\textit{Convergence rates of inertial forward-backward algorithms}.
SIAM Journal on Optimization \textbf{28}(1), 849--874 (2018)

\bibitem{AC2R-EECT}  
\textbf{H. Attouch, A.  Cabot, Z. Chbani, and H. Riahi}, 
{ \it Rate of convergence of inertial gradient dynamics with time-dependent viscous damping coefficient}, Evolution Equations and Control Theory \textbf{7}(3),  353--371 (2018)

\bibitem{ACPR}  
\textbf{H. Attouch,  Z. Chbani, J. Peypouquet, and P. Redont},  
{\it Fast convergence of inertial dynamics and algorithms with asymptotic vanishing viscosity}, 
Mathematical Programming \textbf{168},  123--175 (2018)

\bibitem{ACR-subcrit}  
\textbf{H. Attouch,  Z. Chbani, and H. Riahi},
{\it Rate of convergence  of the Nesterov accelerated gradient method  in the subcritical case  $\alpha \leq 3$}, ESAIM: Control, Optimisation and Calculus of Variations \textbf{25} (2) (2019).

\bibitem{AF}  
\textbf{H. Attouch and J. Fadili}, 
{\it From the Ravine method to the Nesterov method and vice versa: A dynamical system perspective}.
SIAM Journal on Optimization \textbf{32}(3), 2074--2101 (2022)

\bibitem{AP}  
\textbf{H. Attouch and J. Peypouquet},
{\it The rate of convergence of Nesterov's accelerated forward-backward method is actually faster than $1/k^2$}, SIAM Journal on Optimization \textbf{26}(3), 1824--1834 (2016)

\bibitem{APR}
\textbf{H. Attouch, J. Peypouquet, and P. Redont},
\textit{Fast convex optimization via inertial dynamics with Hessian driven damping},
Journal of Differential Equations \textbf{261}(19), 5734--5783 (2016)

\bibitem{AD} 
\textbf{J.-F. Aujol and Ch. Dossal}, 
{\it Stability of over-relaxations for the Forward-Backward
algorithm, application to FISTA},  
SIAM Journal on Optimization \textbf{25}(4),  2408--2433 (2015)

\bibitem{AD17} 
\textbf{J.-F. Aujol and Ch. Dossal}, 
{\it Optimal rate of convergence of an ODE associated to the Fast Gradient Descent schemes for} $b>0$, \url{https://hal.inria.fr/hal-01547251v2} (2017)

\bibitem{BT:09}
\textbf{A. Beck and M. Teboulle},
\textit{A Fast Iterative Shrinkage-Thresholding Algorithm for linear inverse problems},
SIAM Journal on Imaging Sciences \textbf{2}(1), 183--202 (2009)

\bibitem{BCCH}
\textbf{R. I. Bo\c{t}, E. Chenchene, E.R. Csetnek, D. A. Hulett},
\textit{Accelerating Diagonal Methods for Bilevel Optimization: Unified Convergence via Continuous-Time Dynamics}
\url{arXiv:2505.14389} (2025)

\bibitem{BCN:23}
\textbf{R.I. Bo\c{t}, E.R. Csetnek and D.-K. Nguyen},
\textit{Fast Augmented Lagrangian Method in the convex regime with convergence guarantees for the iterates}. 
Mathematical Programming \textbf{200}, 147-–197 (2023)

\bibitem{CD:15}
\textbf{A. Chambolle and C. Dossal}, 
\textit{On the convergence of the iterates of the ``Fast Iterative Shrinkage/Thresholding Algorithm''}. 
Journal of Optimization Theory and Applications \textbf{166}, 968--982 (2015)

\bibitem{GT}   
\textbf{I.M. Gelfand and M.  Tsetlin}, 
{\it Printszip nelokalnogo poiska v sistemah avtomatich}, Doklady Akademii Nauk SSSR \textbf{137}, 295--298 (1961) (in Russian)

\bibitem{JR:25}
\textbf{U. Jang, E.K. Ryu},
\textit{Point convergence of Nesterov’s Accelerated Gradient Method: an AI-assisted proof}, arXiv:2510.23513 (2025)


\bibitem{May}  
\textbf{R. May}, 
{\it Asymptotic for a second-order evolution equation with convex potential and vanishing damping term},  Turkish Journal of Mathematics \textbf{41}(3),  681--685 (2017)

\bibitem{Nesterov} 
\textbf{Y. Nesterov}, 
{\it A method of solving a convex programming problem with convergence rate $O(1/{k^2})$}, 
Doklady Akademii Nauk SSSR  \textbf{27},  372--376 (1983)

\bibitem{Polyak1964} 
\textbf{B.T. Polyak}, 
{\it Some methods of speeding up the convergence of iteration methods}, 
USSR Computational Mathematics and Mathematical Physics \textbf{4},  1--17 (1964)

\bibitem{Ryu}\textbf{E. Ryu}, \url{https://x.com/ernestryu/status/1980759528984686715?s=43} (2025)

\bibitem{SBC}  
\textbf{W. J. Su,  S. Boyd, and  E. J. Cand\`es}, 
{\it A differential equation for modeling Nesterov's accelerated gradient method: theory and insights}. Neural Information Processing Systems \textbf{27} (2014),  2510--2518. 

\end{thebibliography}
\end{document}